\documentclass[10pt]{amsart}

\usepackage{amsmath,amssymb,amsthm}

\theoremstyle{plain}
\newtheorem{theorem}{Theorem}[section]
\newtheorem{proposition}[theorem]{Proposition}
\newtheorem{lemma}[theorem]{Lemma}
\newtheorem{corollary}[theorem]{Corollary}

\theoremstyle{definition}

\theoremstyle{remark}
\newtheorem{remark}[theorem]{Remark}

\DeclareMathOperator{\N}{N}
\newcommand{\OK}{\mathcal O_K}
\newcommand{\F}{\mathcal F}
\newcommand{\Z}{\mathbb Z}
\newcommand{\Q}{\mathbb Q}
\newcommand{\eps}{\varepsilon}
\newcommand{\omegaD}{\omega_D}

\title[Kim's octonary form]{For which real quadratic fields is Kim's octonary form universal?}
\author{Scott Duke Kominers}
\address{Harvard Business School; Department of Economics and Center of Mathematical Sciences and Applications, Harvard University; and a16z crypto}
\email{kominers@fas.harvard.edu}
\thanks{I used LLMs to assist with computations in the preparation of this article, particularly GPT-5.5~Pro and Claude~4.8 Opus (both accessed in part via Poe with the support of Quora, where I am an advisor).  I especially appreciate helpful comments from V\'{\i}t\v{e}zslav Kala, as well as a review from Refine.ink.  The problem, methods, and eventual written form are my own; and of course any errors remain my responsibility.  This work was conducted while I was visiting the Technological Innovation, Entrepreneurship, and Strategic Management (TIES) Group at the MIT Sloan School of Management; I greatly appreciate their hospitality.}

\begin{document}

\begin{abstract}
Let $K=\Q(\sqrt D)$ with $D>1$ squarefree, and let $\eps_+$ be the totally positive fundamental unit of $\OK$.  B.~M.~Kim proved in 2000 that the octonary diagonal form
\[
 f=x_1^2+\cdots+x_4^2+\eps_+(x_5^2+\cdots+x_8^2)
\]
is universal over $\OK$ whenever $D=n^2-1$ is squarefree.  We complete Kim's result to an if-and-only-if classification: $f$ is universal if and only if $D=n^2-1$ for some $n\ge2$, or $D=n^2-4$ for some odd $n\ge3$, in both cases subject to squarefreeness.  The second family appears to be new in this context and contains $\Q(\sqrt5)$ at $n=3$ as a degenerate boundary case, recovering Maa{\ss}'s classical exceptional three-square phenomenon.  Equivalently, $f$ is universal over $\OK$ if and only if the Blomer--Kala invariant $M_D$ equals $1$; for the two stated families we have $M_D=1$, so the Blomer--Kala universal $8M_D$-variable construction specializes exactly to $f$.  The converse leverages a continued-fraction involution $\tau(\gamma)=\eps_+\gamma'$ together with a closed formula in convergent coordinates, a three-consecutive-square-values lemma for primitive quadratic polynomials of positive fundamental discriminant, and an even-root exclusion lemma derived from complete-quotient norm bounds.
\end{abstract}

\maketitle

\section{Introduction}

Let $K=\Q(\sqrt D)$ be a real quadratic field with $D>1$ squarefree, let $\OK$ be its ring of integers, and let $\eps_+$ be the totally positive fundamental unit of $\OK$.  A quadratic form over $\OK$ is said to be \emph{universal} if it represents every totally positive element of $\OK$.  

In 2000, B.~M.~Kim \cite{Kim2000} proved that the octonary diagonal form
\[
 f=x_1^2+x_2^2+x_3^2+x_4^2+\eps_+(x_5^2+x_6^2+x_7^2+x_8^2)
\]
is universal over $\OK$ whenever $D=n^2-1$ is squarefree.  This paper determines exactly the real quadratic fields over which $f$ is universal: Kim's fields form one of precisely two such families, and no others occur.

\begin{theorem}\label{thm:main}
Let $K=\Q(\sqrt D)$ with $D>1$ squarefree, let $\OK$ be its ring of integers, and let $\eps_+$ be the totally positive fundamental unit.  The diagonal form
\[
 f=x_1^2+x_2^2+x_3^2+x_4^2+\eps_+(x_5^2+x_6^2+x_7^2+x_8^2),
 \qquad x_i\in\OK,
\]
is universal over $\OK$ if and only if either
\[
        D=n^2-1\quad(n\ge2),
\]
or
\[
        D=n^2-4\quad(n\ge3\text{ odd}),
\]
with $D$ squarefree.
\end{theorem}

As a diagonal shape over $\OK$, the form $f$ has only one nontrivial coefficient, namely $\eps_+$; Theorem~\ref{thm:main} is thus a statement about which fields admit universality of one particular octonary shape.  It has a clean reformulation through an invariant of Blomer and Kala \cite{BlomerKala2018}.  To each real quadratic field they attach a finite set $S_0$ of indecomposable totally positive integers, taken modulo totally positive units, and they build from it a universal diagonal form in $8|S_0|$ variables.  Writing $M_D=|S_0|$---an explicit continued-fraction sum recalled in Section~\ref{sec:prelim}---their form collapses to precisely Kim's eight-variable $f$ exactly when $M_D=1$.  Our classification may therefore be restated as
\[
        f\text{ is universal over }\OK \iff M_D=1;
\]
indeed, Proposition~\ref{prop:MDone} identifies the condition $M_D=1$ with the two families of Theorem~\ref{thm:main}.  The degenerate boundary $n=3$ of the second family is $D=5$, the field of Maa{\ss}'s classical universal sum of three squares \cite{Maass1941}; thus the second family extends this exceptional behaviour to an infinite one-parameter family of fields.

\subsection*{Background} Universal quadratic forms over number fields have a long history.  Universality extends Lagrange's~\cite{Lagrange1770} four-squares theorem over $\Z$; over real quadratic fields the subject begins with Maa{\ss}, who proved universality of a sum of three squares over $\Q(\sqrt5)$ \cite{Maass1941}, and Siegel \cite{Siegel1945}, who showed that sums of squares are universal over only finitely many number fields.  Chan--Kim--Raghavan~\cite{CKR1996} classified ternary universal forms over real quadratic fields, and Kim~\cite{Kim1999,Kim2000} established both finiteness for septenary diagonal forms and the octonary construction studied here.  Dress--Scharlau~\cite{DressScharlau1982} described the indecomposable totally positive integers of a real quadratic field through the continued fraction of $\sqrt D$, a description on which the systematic theory of Kala~\cite{Kala2016a,Kala2016b} and of Blomer--Kala~\cite{BlomerKala2018} builds, tying universality to such indecomposables and to continued fractions; Yatsyna~\cite{Yatsyna2019} gave rank lower bounds through interlacing polynomials; and Hejda--Kala~\cite{HejdaKala2020} described the additive structure of the semigroup of totally positive integers.  Kala's excellent survey \cite{Kala2023survey} gives a fuller account.

\subsection*{Our contributions} Beyond the classification itself, the paper contributes results of two kinds. As results about $f$, we exhibit the second family $D=n^2-4$ ($n$ odd) and prove universality of $f$ there, and we prove the converse---that no squarefree~$D$ outside the two families admits universality of $f$. As methods, the converse is carried by three tools that nowhere mention the octonary shape $f$---each is a statement about the indecomposables, the convergents, or the continued fraction of an arbitrary real quadratic field, and so we hope that the way we combine them may provide a blueprint for future work. 

\subsection*{Outline of the proof}
Sufficiency (Section~\ref{sec:suff}) is brief: when $M_D=1$ the Blomer--Kala form is $f$ itself, and the elementary two-unit-layer decomposition in Proposition~\ref{prop:sufficiency}, together with Lagrange's theorem, shows that $f$ represents every totally positive integer. The converse occupies the rest of the paper.  Assume $D\notin\{n^2-1:n\ge2,\ n^2-1\text{ squarefree}\}$, equivalently $M_D\ge2$; we must produce a totally positive integer not represented by $f$.

The main organizing fact (Lemma~\ref{lem:indecomp-square}) is that any indecomposable represented by $f$ is either a square or $\eps_+$ times a square.  It therefore suffices to find one indecomposable lying in neither square class.  The shape of that search is dictated by the period length $s$ of the continued fraction of $\omegaD$, through the norm $\N(\eta)=(-1)^s$ of the fundamental unit:
\begin{itemize}
\item If $s$ is odd (Section~\ref{sec:odd}), then $\eps_+=\eta^2$ is already a square, the two square classes coincide, and any non-square indecomposable will do.
\item If $s$ is even, the two classes are distinct and are swapped by the involution $\tau$ of Section~\ref{sec:tau}; the period $s=2$ is dispatched there directly.
\item For $s\ge4$, a non-square \emph{norm} ends the matter at once (Lemma~\ref{lem:norm-easy}); failing that, every indecomposable has square norm, and a norm-polynomial obstruction (Section~\ref{sec:normpoly}) forces the rigid \emph{residual shape} in which every odd-indexed partial quotient equals $1$.
\item The residual shape is removed in Section~\ref{sec:residual} by splitting on $s\bmod4$: a $\tau$-fixed indecomposable when $4\mid s$, an explicit computation when $s=6$, and the even-root exclusion lemma with a minimal-norm argument when $s\equiv2\pmod4$ and $s\ge10$.
\end{itemize}
Section~\ref{sec:proof} assembles the cases just described into a proof of Theorem~\ref{thm:main}.

\subsection*{Methods of broader interest}

The ingredients we use in our converse proof are for the most part classical. What may be useful beyond the present problem is the way they fit together into a procedure for deciding the universality of a fixed diagonal form whose coefficients are units. The procedure has three moves, which we sketch here at the level of generality at which we expect the approach to apply.

The first move trades universality for a question about indecomposables. A fixed unit-coefficient diagonal form represents an indecomposable only when all but one of its squared terms vanish, so it can represent an indecomposable only if that indecomposable lies in one of finitely many unit square classes---here $\OK^2$ and $\eps_+\OK^2$ (Lemma~\ref{lem:indecomp-square}). Universality therefore fails as soon as a single indecomposable is produced outside those classes, and the problem passes from all totally positive integers to the explicitly described indecomposables.

The second move pairs the square classes by an involution. When the relevant classes are distinct, conjugation followed by multiplication by the fundamental unit, $\tau(\gamma)=\eps_+\gamma'$, interchanges them, so an indecomposable with neither $\gamma$ nor $\tau(\gamma)$ a square defeats both classes at once. On the Dress--Scharlau indecomposables $\tau$ is the index reflection of Lemma~\ref{lem:tau-formula}, which recasts the search as combinatorics of the palindromic period; the relevant bookkeeping is controlled by the period length through $\N(\eta)=(-1)^s$ and by the parity of $s/2$.

The third move is a descent on norms. The square-value obstruction of Lemma~\ref{lem:three-square-values}---a primitive quadratic of fundamental discriminant cannot be a perfect square at three consecutive integers---forces the continued fraction into a rigid residual shape once every indecomposable is required to have square norm. To clear that shape in the hardest periods we use the closed identity
\[
        |\N(\alpha_i)|=\frac{\delta}{|\theta_{i+1}-\theta_{i+1}'|}
\]
of Lemma~\ref{lem:complete-quotient}, which expresses the absolute norm of a convergent through its complete quotient. Further bounds that follow power the minimal-norm descent that rules out squares among the convergents.

We expect the strategy just described---reduce to square classes of indecomposables, pair the classes by $\tau$, and control the survivors through a norm descent---to be the part of the argument most likely to transfer: for example, to other fixed diagonal shapes, and to fields of larger $M_D$, where the same square-class bookkeeping applies with more classes in play.

\section{Preliminaries}\label{sec:prelim}

We let
\[
\omegaD=\begin{cases}
\sqrt D &D\equiv2,3\pmod4,\\[2mm]
(1+\sqrt D)/2 &D\equiv1\pmod4.
\end{cases}
\]
Then $\OK=\Z[\omegaD]$.  We denote conjugation by a prime, and we write $\OK^+$ for the set of totally positive elements of $\OK$.  We set
\[
        \delta=\omegaD-\omegaD'=\sqrt{\Delta_K},
\]
where $\Delta_K=4D$ if $D\equiv2,3\pmod4$ and $\Delta_K=D$ if $D\equiv1\pmod4$.

Write the continued fraction of $\omegaD$ in the form
\[
        \omegaD=[u_0;\overline{u_1,u_2,\ldots,u_s}],
\]
where the period length $s$ is minimal. (The subscripts $u_j$ are extended periodically for $j\ge1$.) 
Thus we have
\[
        u_s=\begin{cases}
        2u_0&D\equiv2,3\pmod4,\\
        2u_0-1&D\equiv1\pmod4,
        \end{cases}
\]
and $u_i=u_{s-i}$ for $1\le i\le s-1$.  Let
\[
        \frac{p_i}{q_i}=[u_0;u_1,\ldots,u_i]
\]
be the $i$-th convergent, with $p_{-1}=1$, $q_{-1}=0$, and put
\[
        \alpha_i=p_i-q_i\omegaD'\qquad(i\ge -1).
\]
Then $\alpha_{-1}=1$, and
\[
        \alpha_{i+1}=u_{i+1}\alpha_i+\alpha_{i-1}.
\]
For $i\ge -1$ and $0\le r\le u_{i+2}$ define the semiconvergent
\[
        \alpha_{i,r}=\alpha_i+r\alpha_{i+1};
\]
thus, $\alpha_{i,u_{i+2}}=\alpha_{i+2}$.

We shall use the following standard description of indecomposables, due to Dress--Scharlau~\cite[Theorem 2]{DressScharlau1982} and used in the form stated by Blomer--Kala~\cite[Section 2]{BlomerKala2018} and Hejda--Kala~\cite[Section 2]{HejdaKala2020}.

\begin{theorem}[Dress--Scharlau, Theorem 2]\label{thm:indecomp}
The indecomposable totally positive elements of $\OK$ are precisely
\[
        \alpha_{i,r},\quad \alpha_{i,r}'
\]
with $i\ge -1$ odd and $0\le r<u_{i+2}$.  Moreover, $\alpha_i$ is totally positive if and only if $i$ is odd, with $-1$ regarded as odd.
\end{theorem}

Let $\eta>1$ be the fundamental unit of $\OK$, not necessarily totally positive.  Then $\eta=\alpha_{s-1}$ and $\N(\eta)=(-1)^s$.  The totally positive fundamental unit is
\[
        \eps_+=\begin{cases}
        \eta &s\text{ even},\\
        \eta^2 &s\text{ odd}.
        \end{cases}
\]

Blomer--Kala define
\[
M_D=\begin{cases}
 u_1+u_3+\cdots+u_{s-1} &s\text{ even},\\
 2u_0+u_1+\cdots+u_{s-1} &s\text{ odd},\;D\equiv2,3\pmod4,\\
 2u_0+u_1+\cdots+u_{s-1}-1 &s\text{ odd},\;D\equiv1\pmod4.
\end{cases}
\]
Equivalently, if $s$ is odd then $M_D=u_1+\cdots+u_s$.  The role of $M_D$ is the following.  The indecomposable totally positive integers, taken modulo multiplication by totally positive units, fall into finitely many classes; Blomer--Kala show that there are exactly $M_D$ of them, with representatives forming the set $S_0$ \cite[Section 2]{BlomerKala2018}.  Universality of a diagonal form is governed by which of these classes it represents, and Blomer--Kala assemble a universal form by devoting, to each class $\sigma\in S_0$, eight variables: four with coefficient $\sigma$ and four with coefficient $\eps_+\sigma$.  This yields a universal diagonal form in $8M_D$ variables \cite[Theorem 10]{BlomerKala2018}.  When $M_D=1$ the only class is that of $1$, the eight coefficients are $1$ (four times) and $\eps_+$ (four times), and the form is exactly $f$.  The content of the converse half of this paper is that $M_D=1$ is also \emph{necessary} for this particular eight-variable shape to be universal.

We shall also use the following square-root restriction, which limits the possible square roots of an indecomposable to the convergents $\alpha_j$ and their conjugates.

\begin{lemma}[Blomer--Kala, Lemma 11]\label{lem:BK11}
If $v\in\OK$ and $v^2$ is indecomposable, then for some $j\ge -1$,
\[
        v\in\{\pm\alpha_j,\pm\alpha_j'\}.
\]
\end{lemma}

\section{The cases with $M_D=1$}\label{sec:MDone}

Both halves of Theorem~\ref{thm:main} pass through the condition $M_D=1$---so before proving anything about $f$ we identify the fields that satisfy it.  The invariant $M_D$ is a sum of partial quotients of $\omegaD$ over one period---the odd-indexed ones when the period is even---so $M_D=1$ is the most rigid possible demand on that expansion.  Forcing the sum down to $1$ leaves only the shortest periods: period-$2$ with its single free partial quotient equal to $1$, together with the degenerate period-$1$ expansion of $\Q(\sqrt5)$.  Solving the resulting quadratic for the periodic tail then produces exactly two infinite families, one for each parity class of $D$.

Let
\[
\F=\{n^2-1:n\ge2,\ n^2-1\text{ squarefree}\}
\cup
\{n^2-4:n\ge3\text{ odd},\ n^2-4\text{ squarefree}\}.
\]

\begin{proposition}\label{prop:MDone}
For squarefree $D>1$, we have $M_D=1$ if and only if $D\in\F$.
\end{proposition}

\begin{proof}
First suppose $D=n^2-1$ is squarefree.  Squarefreeness forces $n$ to be even: if $n\ge3$ is odd, then $n^2-1=(n-1)(n+1)$ is divisible by $4$ and hence is not squarefree.  Also
\[
        \sqrt{n^2-1}=[n-1;\overline{1,2n-2}].
\]
Solving the quadratic equation defined by the displayed purely periodic tail gives this continued fraction and no shorter period.  Thus $s=2$ and $u_1=1$, so $M_D=1$.

Next suppose $D=n^2-4$ with $n=2m+1$ odd.  Then $D\equiv1\pmod4$.  For $n=3$ we have $D=5$ and
\[
        \omegaD=[1;\overline{1}],
\]
so $M_D=1$.  For $n\ge5$,
\[
        \omegaD=\frac{1+\sqrt{n^2-4}}2=[m;\overline{1,2m-1}],
\]
again by solving the associated quadratic equation for the period-$2$ expansion.  Hence $s=2$, $u_1=1$, and $M_D=1$.

Conversely, assume $M_D=1$.  If $s$ is odd, then $M_D=u_1+\cdots+u_s$, whence $s=1$ and $u_1=1$.  Since $u_s=2u_0$ in the case $D\equiv2,3\pmod4$, this is impossible there.  Hence $D\equiv1\pmod4$, $u_0=1$, and $D=5=3^2-4$.

If $s$ is even, then
\[
        1=M_D=u_1+u_3+\cdots+u_{s-1},
\]
so $s=2$ and $u_1=1$.  In the case $D\equiv2,3\pmod4$, solving
\[
        \sqrt D=[u_0;\overline{1,2u_0}]
\]
gives $D=u_0^2+2u_0=(u_0+1)^2-1$.  In the case $D\equiv1\pmod4$, solving
\[
        \omegaD=[u_0;\overline{1,2u_0-1}]
\]
gives $D=(2u_0+1)^2-4$.  These are exactly the two asserted families.
\end{proof}

The two families thus correspond to the two arithmetic types of real quadratic field: among fields with $M_D=1$, those with $D\equiv2,3\pmod4$ are exactly $D=n^2-1$, while those with $D\equiv1\pmod4$ are exactly $D=n^2-4$ with $n$ odd.  The boundary case $n=3$, where the period-$2$ expansion degenerates to the period-$1$ golden-section expansion of $\Q(\sqrt5)$, is precisely Maa{\ss}'s field; the second family should be read as the natural continuation of that single classical example into an infinite series.  For illustration, the squarefree examples with $D\le100$ are:
\[
\begin{array}{c|c|c}
D&\text{family}&\text{continued fraction of $\omegaD$}\\ \hline
3&2^2-1&[1;\overline{1,2}]\\
5&3^2-4&[1;\overline{1}]\\
15&4^2-1&[3;\overline{1,6}]\\
21&5^2-4&[2;\overline{1,3}]\\
35&6^2-1&[5;\overline{1,10}]\\
77&9^2-4&[4;\overline{1,7}]
\end{array}
\]

\section{Sufficiency}\label{sec:suff}

The sufficiency direction is short because, once $M_D=1$, the Blomer--Kala machinery already contains the answer: their general $8M_D$-variable form is precisely Kim's~$f$.  Conceptually, $M_D=1$ says that, up to totally positive units, there is only one indecomposable to represent, namely $1$ itself; representing every totally positive integer then reduces to the rational four-square theorem applied one unit-layer at a time.

\begin{proposition}\label{prop:sufficiency}
If $D\in\F$, then the form $f$ in Theorem~\ref{thm:main} is universal over~$\OK$.
\end{proposition}

\begin{proof}
By Proposition~\ref{prop:MDone}, $D\in\F$ implies $M_D=1$.  We give a direct proof of universality in the two families.

First note that in both families $\{1,\eps_+\}$ is a $\Z$-basis of $\OK$.  Indeed, if $D=n^2-1$, then
\[
        \eps_+=n+\sqrt D=n+\omegaD,
\]
so $\omegaD=\eps_+-n$.  If $D=n^2-4$ with $n=2m+1$, then, including the case $D=5$,
\[
        \eps_+=\frac{n+\sqrt D}{2}=m+\omegaD,
\]
so $\omegaD=\eps_+-m$.  Hence $\{\eps_+^k,\eps_+^{k+1}\}$ is a $\Z$-basis of $\OK$ for every $k\in\Z$.

Let $\alpha\in\OK^+$.  Choose $k\in\Z$ such that
\[
        \eps_+^{2k}\le \frac{\alpha}{\alpha'}\le \eps_+^{2k+2}.
\]
Write
\[
        \alpha=a\eps_+^k+b\eps_+^{k+1}
\]
with $a,b\in\Z$.  Put $A=\alpha\eps_+^{-k}$ and $A^*=\alpha'\eps_+^k$.  Then
\[
        A=a+b\eps_+,\qquad A^*=a+b\eps_+^{-1},
\]
so
\[
        b=\frac{A-A^*}{\eps_+-\eps_+^{-1}},
        \qquad
        a=\frac{\eps_+A^*-\eps_+^{-1}A}{\eps_+-\eps_+^{-1}}.
\]
The choice of $k$ gives $A\ge A^*$ and $\eps_+A^*\ge \eps_+^{-1}A$, so $a,b\ge0$.

By Lagrange's theorem, write
\[
        a=r_1^2+\cdots+r_4^2,\qquad
        b=s_1^2+\cdots+s_4^2
\]
with $r_i,s_i\in\Z$.  If $k=2m$, take
\[
        x_i=\eps_+^m r_i\quad(1\le i\le4),
        \qquad
        x_{4+i}=\eps_+^m s_i\quad(1\le i\le4).
\]
If $k=2m+1$, take
\[
        x_i=\eps_+^{m+1}s_i\quad(1\le i\le4),
        \qquad
        x_{4+i}=\eps_+^m r_i\quad(1\le i\le4).
\]
In either case the form $f$ represents $\alpha$.  Thus $f$ is universal.
\end{proof}

\begin{remark}
The direct construction presented in the proof of Proposition~\ref{prop:sufficiency} is exactly the $M_D=1$ specialization of the Blomer--Kala $8M_D$-variable form \cite[Theorem 10]{BlomerKala2018}.
\end{remark} 

\section{Necessity: the square-class dichotomy}\label{sec:dichotomy}

We now turn to the converse, which is the heart of the paper, and we argue by contrapositive.  Assume $D\notin\F$; equivalently, by Proposition~\ref{prop:MDone}, $M_D\ge2$.  We shall exhibit an indecomposable totally positive element of $\OK$ which lies in neither square class $\OK^2$ nor $\eps_+\OK^2$; Lemma~\ref{lem:indecomp-square} then shows that this element is not represented by $f$, and hence that $f$ is not universal.

The strategy is to shift attention from arbitrary totally positive integers to \emph{indecomposable} ones.  The reason this helps is structural: $f$ is a sum of eight squared terms, four with coefficient $1$ and four with coefficient $\eps_+$, and each nonzero such term is totally positive.  An indecomposable, by definition, cannot be split as a sum of two totally positive integers; so if $f$ represents an indecomposable, then all but one of the eight terms must vanish.  This collapses the eight-dimensional representation problem onto a single squared term and yields the key fact that drives everything that follows.

\begin{lemma}\label{lem:indecomp-square}
If an indecomposable $\gamma\in\OK^+$ is represented by $f$, then
\[
        \gamma=v^2\quad\text{or}\quad \gamma=\eps_+v^2
\]
for some nonzero $v\in\OK$.
\end{lemma}

\begin{proof}
Each nonzero summand $x_i^2$ or $\eps_+x_i^2$ is totally positive.  If two nonzero summands occurred, $\gamma$ would decompose as a sum of two nonzero totally positive algebraic integers.  Thus exactly one square is nonzero.
\end{proof}

In other words, the only indecomposables $f$ can possibly miss are those lying outside both square classes $\OK^2$ and $\eps_+\OK^2$---and to obstruct universality, it is enough to exhibit a single such indecomposable.  The simplest way an indecomposable can fall outside both classes is recorded by its norm.

\begin{lemma}\label{lem:norm-easy}
If some indecomposable $\gamma\in\OK^+$ has $|\N(\gamma)|$ not a rational square, then $f$ is not universal.
\end{lemma}

\begin{proof}
If $\gamma=v^2$ or $\gamma=\eps_+v^2$, then $|\N(\gamma)|=\N(v)^2$, since $\N(\eps_+)=1$.  This is a rational square, so an indecomposable with non-square absolute norm can be neither, and Lemma~\ref{lem:indecomp-square} denies its representation.
\end{proof}

The norm test in Lemma~\ref{lem:norm-easy} is blunt: it sees only the product $\gamma\gamma'$ and cannot distinguish the two square classes from one another.  When it fails---that is, when every indecomposable has a square norm---we need a finer instrument that knows which class an element belongs to: the mechanism that supplies this is the norm of the fundamental unit, $\N(\eta)=(-1)^s$, where $s$ is the period of $\omegaD$.  When $s$ is odd, $\eps_+=\eta^2$ is a square and the two classes coincide; when $s$ is even, they are genuinely distinct.  The parity split just described organizes the entire converse argument, and we treat the two cases in turn, beginning with the easier odd case.

\section{The case of $s$ odd}\label{sec:odd}

When $s$ is odd, the fundamental unit $\eta$ has norm $-1$ and $\eps_+=\eta^2$, so
\[
        \eps_+\OK^2=\eta^2\OK^2=(\eta\OK)^2=\OK^2;
\]
that is, the two square classes coincide.  This is the favourable case: multiplying by~$\eps_+$ no longer moves an element between classes, so Lemma~\ref{lem:indecomp-square} reduces universality to the single condition that every indecomposable must be a square---and any one non-square indecomposable provides an obstruction.  In this context, a non-square is straightforward to identify: an element whose $\omegaD$-coordinate equals~$1$ is essentially never a square.  When $D\equiv2,3\pmod4$ a square has even $\omegaD$-coordinate, so it cannot be $1$; when $D\equiv1\pmod4$ the only square with $\omegaD$-coordinate $1$ is $\omegaD^2$ itself, and matching it against our witness pins $D$ to the single excluded value~$5$.

\begin{proposition}\label{prop:negative-unit}
If $s$ is odd and $D\notin\F$, then $f$ is not universal.
\end{proposition}

\begin{proof}
If $u_1\ge2$, take $\gamma=1+\alpha_0=\alpha_{-1,1}$.  If $u_1=1$, then $D\notin\F$ forces $s\ge3$, and we take $\gamma=\alpha_1=1+\alpha_0$.  In either case $\gamma$ is indecomposable with $\omegaD$-coefficient equal to $1$.

If $D\equiv2,3\pmod4$, a square $(a+b\omegaD)^2$ has $\omegaD$-coefficient $2ab$, never $1$.  If $D\equiv1\pmod4$, then
\[
        (a+b\omegaD)^2=a^2+b^2\frac{D-1}{4}+b(2a+b)\omegaD.
\]
The equation $b(2a+b)=1$ forces $(a,b)=(0,1)$ or $(0,-1)$; both choices have square $\omegaD^2=(D-1)/4+\omegaD$.  Equating with $\gamma=u_0+\omegaD$ forces $u_0=(D-1)/4$.  Combined with $u_0=\lfloor(1+\sqrt D)/2\rfloor$, this two-sided inequality
\[
        \frac{D-1}{4}\le\frac{1+\sqrt D}{2}<\frac{D+3}{4}
\]
gives, after squaring and simplification, $(D-1)(D-9)\le0$ and $(D-1)^2>0$, hence $1<D\le9$.  Among $D\equiv1\pmod4$ squarefree in this range only $D=5$ remains, and $D=5\in\F$.  Thus $\gamma$ is not a square.  Since $s$ is odd, the two square classes coincide---so $\gamma$ lies in neither, and Lemma~\ref{lem:indecomp-square} shows that $f$ is not universal.
\end{proof}

\section{The involution $\tau$ and the case of $s=2$}\label{sec:tau}

For even $s$ the two square classes no longer coincide, and a single non-square witness is no longer enough on its own: an indecomposable could fail to be a square yet still be $\eps_+$ times a square.  What rescues the argument is a symmetry that exchanges the two classes, so that one witness can be made to testify against both at once.

We assume henceforth that $s$ is even, so $\eps_+=\eta=\alpha_{s-1}$ has norm $1$ and the two square classes $\OK^2$ and $\eps_+\OK^2$ are now distinct.  Write $\eps=\eps_+$ and define
\[
        \tau(\gamma)=\eps\gamma'.
\]
Since $\N(\eps)=1$ we have $\tau^2=\mathrm{id}$, and $\tau$ interchanges the two square classes:
\[
        \tau(v^2)=\eps(v')^2,
        \qquad
        \tau(\eps v^2)=(v')^2.
\]
Thus an element is $\eps$ times a square if and only if its image under $\tau$ is a square.  This is the device that lets one non-square witness exclude \emph{both} classes at once: if neither~$\gamma$ nor $\tau(\gamma)$ is a square, then $\gamma$ is in neither class.  To use it on the indecomposables, which the Dress--Scharlau description lists as convergents and semiconvergents, we need to know how $\tau$ acts on those coordinates. The palindromicity of the period makes the action a clean index reflection; we record this fact here in the normalization we use in the sequel.

\begin{lemma}\label{lem:tau-formula}
For $-1\le i\le s-1$,
\[
        \eps\alpha_i'=(-1)^{i+1}\alpha_{s-i-2}.
\]
Consequently, for odd $i$ with $-1\le i\le s-3$ and $1\le r\le u_{i+2}-1$,
\begin{gather*}
        \tau(\alpha_{i,r})=\alpha_{s-i-4,\,u_{i+2}-r},\\
        \tau(\alpha_{i,0})=\alpha_{s-i-2,0}.
\end{gather*}
For $1\le r\le u_{i+2}-1$, the semiconvergent formula stays within the fundamental block of indices appearing in $S_0$, since $s-i-4\in[-1,s-3]$.  For $r=0$, the endpoint $i=-1$ gives $\tau(\alpha_{-1,0})=\alpha_{s-1,0}=\eps$, the totally positive fundamental unit, as expected.
\end{lemma}

\begin{proof}
For $i=s-1$ the identity is $\eps\alpha_{s-1}'=\N(\eps)=1=\alpha_{-1}$.  For $i=s-2$, let $\theta_s=[u_s;u_1,u_2,\ldots]$ be the complete quotient just before the period closes.  Since $\omegaD=u_0+1/[u_1;u_2,\ldots]$, we have $\theta_s=u_s+\omegaD-u_0=\alpha_0$.  The standard identity $\theta_s=-\alpha_{s-2}'/\alpha_{s-1}'$ gives $\eps\alpha_{s-2}'=-\alpha_0$, which is the asserted formula because $s$ is even.  Equivalently, this is the $i=s-2$ coordinate form of the determinant identity
\[
        p_{i+1}q_i-p_iq_{i+1}=(-1)^i.
\]
For the inductive step, assume the formula holds for $i+1$ and $i+2$.  Conjugating the recurrence $\alpha_{i+2}=u_{i+2}\alpha_{i+1}+\alpha_i$ gives $\alpha_i'=\alpha_{i+2}'-u_{i+2}\alpha_{i+1}'$, so
\[
\begin{aligned}
        \eps\alpha_i'
        &=\eps\alpha_{i+2}'-u_{i+2}\,\eps\alpha_{i+1}'\\
        &=(-1)^{i+3}\alpha_{s-i-4}-u_{i+2}(-1)^{i+2}\alpha_{s-i-3}\\
        &=(-1)^{i+1}\bigl[\alpha_{s-i-4}+u_{i+2}\alpha_{s-i-3}\bigr].
\end{aligned}
\]
By palindromicity $u_{i+2}=u_{s-i-2}$, and the forward recurrence gives $\alpha_{s-i-2}=u_{s-i-2}\alpha_{s-i-3}+\alpha_{s-i-4}$.  Therefore $\alpha_{s-i-4}+u_{i+2}\alpha_{s-i-3}=\alpha_{s-i-2}$, completing the induction.

For $i$ odd,
\[
\begin{aligned}
\tau(\alpha_{i,r})
 &=\eps\alpha_i'+r\eps\alpha_{i+1}'  \\
 &=\alpha_{s-i-2}-r\alpha_{s-i-3}  \\
 &=\alpha_{s-i-4}+(u_{i+2}-r)\alpha_{s-i-3},
\end{aligned}
\]
which is the asserted semiconvergent formula for $r\ge1$.  The case $r=0$ is immediate.
\end{proof}

With $\tau$ in hand, the shortest even period is already within reach.  Here $\tau$ permutes the period-$2$ semiconvergents $\gamma_r$ among themselves by $r\mapsto u_1-r$, so once no $\gamma_r$ is a square, the swap property places every $\gamma_r$ outside both classes at once.

\begin{proposition}\label{prop:s2}
If $s=2$, $u_1\ge2$, and $D\notin\F$, then $f$ is not universal.
\end{proposition}

\begin{proof}
For $1\le r\le u_1-1$, put
\[
        \gamma_r=1+r\alpha_0=\alpha_{-1,r};
\]
these are indecomposable by Theorem~\ref{thm:indecomp}.
A direct calculation from Lemma~\ref{lem:tau-formula} gives
\[
        \tau(\gamma_r)=\gamma_{u_1-r}.
\]
We show that no $\gamma_r$ is a square.

First let $D\equiv2,3\pmod4$.  Then
\[
        \gamma_r=(1+ru_0)+r\omegaD.
\]
If $\gamma_r=(c+d\omegaD)^2$, then
\[
        r=2cd,
        \qquad
        1+ru_0=c^2+Dd^2.
\]
Eliminating $r$ gives
\[
        (c-u_0d)^2+(D-u_0^2)d^2=1.
\]
Since $u_0=\lfloor\sqrt D\rfloor$, the coefficient $D-u_0^2$ is a positive integer.  If $D-u_0^2=1$, the continued fraction would have period $1$; since $s=2$, we have $D-u_0^2\ge2$.  Thus the equality above forces $d=0$, and then $r=0$, a contradiction.

Now let $D\equiv1\pmod4$.  Then
\[
        \alpha_0=(u_0-1)+\omegaD,
        \qquad
        \gamma_r=(1+r(u_0-1))+r\omegaD.
\]
If $\gamma_r=(c+d\omegaD)^2$, then
\[
        r=d(2c+d),
\]
and
\[
        1+r(u_0-1)=c^2+d^2\frac{D-1}{4}.
\]
Thus
\[
        (c-(u_0-1)d)^2+\frac{D-(2u_0-1)^2}{4}d^2=1.
\]
Again, since $u_0=\lfloor(1+\sqrt D)/2\rfloor$, the second coefficient is a positive integer.  If it were $1$, the continued fraction would have period $1$; since $s=2$, it is at least $2$.  Thus the equality above forces $d=0$, and then $r=0$, contradiction.

Therefore no $\gamma_r$ is a square.  Since $\tau$ permutes the set $\{\gamma_r:1\le r\le u_1-1\}$, no $\gamma_r$ is $\eps$ 
times a square either.  Lemma~\ref{lem:indecomp-square} thus gives an obstruction to universality.
\end{proof}

\section{The case of even $s\ge 4$}\label{sec:sgeq4}

In this section, we consider $s$ even with $s\ge4$, and $D\notin\F$. 

\subsection{A norm-polynomial obstruction}\label{sec:normpoly}

If some indecomposable has non-square absolute norm, then Lemma~\ref{lem:norm-easy} already finishes the argument.  The substance of this section is the alternative, rigid scenario: if \emph{every} indecomposable has square absolute norm, then the continued fraction of $\omegaD$ is forced into a very specific form.

The intuition is that ``square norm at every indecomposable'' is an enormous overdetermination.  Along a single odd index $i$, the norms $\N(\alpha_i+r\alpha_{i+1})$ are the values of one integer quadratic polynomial $Q_i(r)$, and the run $r=0,1,\dots,u_{i+2}-1$ records consecutive indecomposables.  Requiring all of these to be perfect squares asks a fixed quadratic to hit squares at many consecutive integers.  The discriminant of $Q_i$ is the field discriminant $\Delta_K$, a fundamental---hence non-square---discriminant. 

How many consecutive integer values of a non-square quadratic can be perfect squares is a classical question. Allison~\cite{Allison1986} and Bremner~\cite{Bremner2003} exhibited quadratics taking square values at eight and seven consecutive integers, and Gonz\'alez-Jim\'enez--Xarles~\cite{GJX2011} settled the symmetric case, none of these placing any restriction on the discriminant.  The next lemma is a refinement of that circle of results tailored to our setting: when the discriminant is a positive \emph{fundamental} discriminant, not even three consecutive square values can occur.

\begin{lemma}[No three consecutive square values]\label{lem:three-square-values}
Let
\[
        Q(x)=Ax^2+Bx+C\in\Z[x]
\]
be primitive, and suppose that its discriminant
\[
        \Delta=B^2-4AC
\]
is a positive fundamental discriminant.  Then $Q$ cannot take perfect-square values at three consecutive integers.
\end{lemma}

\begin{proof}
Only the fundamental-discriminant hypothesis is used in the argument; primitivity is retained in the statement because it is the standard binary-form setting supplied by the application.  After translating $x$, it is enough to rule out the possibility that $Q(0)$, $Q(1)$, and $Q(2)$ are all squares.

If $\Delta\equiv1\pmod4$, then $B$ is odd.  Since
\[
        Q(2)-Q(0)=4A+2B\equiv2\pmod4,
\]
$Q(2)$ is congruent to $Q(0)+2$ modulo $4$, impossible for two squares.

It remains to consider $\Delta=4D$, where $D$ is squarefree and $D\equiv2$ or $3\pmod4$.  Write $B=2b$, so that
\[
        b^2-AC=D.
\]
Since $Q(0)=C$ is a square, either $C$ is odd (and then $C\equiv1\pmod8$), or $4\mid C$.  If $4\mid C$, then $b^2-AC\equiv b^2\pmod4$, which lies in $\{0,1\}$ modulo $4$; this is incompatible with $D\equiv2,3\pmod4$.  Hence $C$ is odd, with $C\equiv1\pmod8$.

If $D\equiv2\pmod4$, then the relation $b^2-AC\equiv2\pmod4$ with $C\equiv1\pmod4$ forces:  if $b$ is even, $-A\equiv2\pmod4$, hence $A\equiv2\pmod4$, and $Q(1)=A+2b+C\equiv3\pmod4$; if $b$ is odd, $1-A\equiv2\pmod4$, hence $A\equiv3\pmod4$, and $Q(1)\equiv2\pmod4$.  Both are impossible for squares.

Finally suppose $D\equiv3\pmod4$.  Then $b^2-AC\equiv3\pmod4$ with $C\equiv1\pmod4$.  If $b$ is even, then $-A\equiv3\pmod4$, so $A\equiv1\pmod4$, and $Q(1)\equiv2\pmod4$, impossible.  Hence
\[
        b\equiv1\pmod2,
        \qquad
        A\equiv2\pmod4.
\]
Then $Q(0),Q(1),Q(2)$ are odd squares, so all are congruent to $1\pmod8$.  Therefore we have
\begin{gather*}
        Q(1)-Q(0)=A+2b\equiv0\pmod8\\
        Q(2)-Q(1)=3A+2b\equiv0\pmod8.
\end{gather*}
Subtracting the preceding identities gives $2A\equiv0\pmod8$, i.e., $A\equiv0\pmod4$, which contradicts the earlier conclusion that $A\equiv2\pmod4$.
\end{proof}

Lemma~\ref{lem:three-square-values} forces every odd-indexed partial quotient down to $1$.

\begin{corollary}\label{cor:odd-u-one}
Assume $s$ is even and $D\notin\F$.  If all indecomposable totally positive elements have square absolute norm, then
\[
        u_1=u_3=\cdots=u_{s-1}=1.
\]
\end{corollary}

\begin{proof}
We fix an odd $i$; for the first-period conclusion it is enough to take
$-1\le i\le s-3$.  If $u_{i+2}\ge2$, then the elements
\[
        \alpha_i,
        \quad
        \alpha_i+\alpha_{i+1},
        \quad
        \alpha_i+2\alpha_{i+1}
\]
are indecomposable; the last is the endpoint $\alpha_{i+2,0}$ when $u_{i+2}=2$.  Consider
\[
        Q_i(r)=\N(\alpha_i+r\alpha_{i+1}).
\]
The pair $\{\alpha_i,\alpha_{i+1}\}$ is a $\Z$-basis of $\OK$: the determinant identity \[p_iq_{i+1}-p_{i+1}q_i=\pm1\] shows that the change-of-basis matrix from $\{\alpha_i,\alpha_{i+1}\}$ to $\{1,\omegaD\}$ has determinant $\pm1$.  The binary norm form in the basis $\{1,\omegaD\}$ has discriminant $\Delta_K$; this unimodular change of basis preserves the discriminant, so
\[
        F_i(X,Y)=\N(X\alpha_i+Y\alpha_{i+1})
\]
has discriminant $\Delta_K$, and hence $Q_i(r)=F_i(1,r)$ has discriminant $\Delta_K$.  Moreover $F_i$ represents $1$ at the integer pair $(X,Y)$ with $X\alpha_i+Y\alpha_{i+1}=1$, so its three coefficients have gcd $1$.  Therefore $Q_i(r)=F_i(1,r)$ is a primitive integer polynomial in $r$.  Lemma~\ref{lem:three-square-values} shows that $Q_i(0)$, $Q_i(1)$, and $Q_i(2)$ cannot all be squares.  Hence $u_{i+2}=1$ for every such odd $i$, as claimed.
\end{proof}

\subsection{Residual positive-unit cases}\label{sec:residual}

In addition to our standing assumptions that~$s$ is even, $s\ge 4$, and $D\notin\F$, throughout the remainder of this section we assume that every indecomposable totally positive element of $\OK$ has square absolute norm. By Corollary~\ref{cor:odd-u-one}, our hypotheses force $u_1=u_3=\cdots=u_{s-1}=1$; we refer to this as the \emph{residual shape}.

What remains is not to rule out the residual shape itself, which can occur, but to show that even in this residual shape one still finds an indecomposable outside the two square classes allowed by Lemma~\ref{lem:indecomp-square}.  Here the period length re-enters---now through $\tau$ rather than through $\N(\eta)$.  Setting $\beta_j=\alpha_{2j-1}$, the involution acts on the odd-indexed convergents by the reflection $\tau(\beta_j)=\beta_{t-j}$, with $t=s/2$.  Whether the reflection has a fixed point depends on the parity of $t$, i.e., on $s\bmod4$, and that dichotomy dictates the argument.  When $4\mid s$ the reflection fixes $\beta_{t/2}$, and a fixed point furnishes an obstruction immediately; when $s\equiv2\pmod4$ there is no fixed point, and we must instead play the convergents off against one another by size, after disposing of the smallest period $s=6$ by direct computation.

\subsubsection{The case $s\equiv0\pmod4$}

Here $t=s/2$ is even, so the reflection $j\mapsto t-j$ fixes $j=t/2$.  The corresponding indecomposable $\beta_{t/2}$ is its own image under $\tau$, and an element fixed by $\tau$ cannot be a square without forcing a totally positive unit strictly between $1$ and $\eps$---which would contradict the fundamentality of $\eps$.  This single observation settles the whole case, and indeed uses only $s\equiv0\pmod4$, not the full residual shape.

Let $t=s/2$ and define
\[
        \beta_j=\alpha_{2j-1}\qquad(0\le j\le t).
\]
Then Lemma~\ref{lem:tau-formula} gives
\[
        \tau(\beta_j)=\beta_{t-j}.
\]

\begin{proposition}\label{prop:s0mod4}
Under the standing assumptions of Section~\ref{sec:residual}, if $s\equiv0\pmod4$, then $f$ is not universal.
\end{proposition}

\begin{proof}
Here $t$ is even.  The element $\beta_{t/2}=\alpha_{s/2-1}$ is indecomposable by Theorem~\ref{thm:indecomp} and is fixed by $\tau$.  Suppose first that $\beta_{t/2}=v^2$.  Since $\tau(\beta_{t/2})=\beta_{t/2}$,
\[
        v^2=\eps(v')^2,
\]
so we have
\[
        \left(\frac{v}{v'}\right)^2=\eps.
\]
The element $w=v/v'\in K$ satisfies the monic polynomial $X^2-\eps$ with coefficients in $\OK$, so $w$ is integral over $\OK$; since $\OK$ is integrally closed in $K$, we have $w\in\OK$.  Moreover $\N(w)=\N(v)/\N(v')=1$, so $w$ is a unit of $\OK$.  Replacing $w$ by $-w$ if necessary, $w^2=\eps$ gives a totally positive unit whose real value lies strictly between $1$ and $\eps$, contradicting the fundamentality of $\eps=\eps_+$.  If $\beta_{t/2}=\eps v^2$, then applying $\tau$ reduces to the same square case.  Therefore $\beta_{t/2}$ is neither a square nor $\eps$ times a square, blocking universality of $f$ by Lemma~\ref{lem:indecomp-square}.

Note that this argument only uses $s\equiv0\pmod4$, not the full residual shape.
\end{proof}

\subsubsection{The period-$6$ case}

When $s\equiv2\pmod4$ the reflection has no fixed point, so the clean argument used in the case of $s\equiv0\pmod4$ is unavailable.  The smallest such period, $s=6$, is small enough that the residual shape leaves only one free parameter besides $u_0$, and we can simply compute the convergents and confront the candidate square $\alpha_3$ with the unit equation directly.  This brute-force base case also previews the size-based mechanism used for larger periods: a would-be square root must be a convergent, and matching $\omegaD$-coordinates pins it to a tiny set of possibilities, each of which lands $D$ back inside $\F$ or contradicts the unit equation.

\begin{proposition}\label{prop:s6}
Under the standing assumptions of Section~\ref{sec:residual}, if $s=6$, then $f$ is not universal.
\end{proposition}

\begin{proof}
The residual shape is
\[
        \omegaD=[u_0;\overline{1,a,1,a,1,u_6}],
\]
where $u_6=2u_0$ if $D\equiv2,3\pmod4$ and $u_6=2u_0-1$ if $D\equiv1\pmod4$.  Direct computation from the recurrence $p_{i+1}=u_{i+1}p_i+p_{i-1}$, $q_{i+1}=u_{i+1}q_i+q_{i-1}$ gives
\[
\begin{array}{c|c|c}
 i&p_i&q_i\\ \hline
0&u_0&1\\
1&u_0+1&1\\
2&(a+1)u_0+a&a+1\\
3&(a+2)u_0+a+1&a+2\\
4&(a^2+3a+1)u_0+a^2+2a&a^2+3a+1\\
5&(a+1)(a+3)u_0+a^2+3a+1&(a+1)(a+3).
\end{array}
\]
The factorizations $p_5-1=(a+3)p_2$ and $p_5+1=(a+1)[(a+3)u_0+(a+2)]$ make the unit equation $\N(\alpha_5)=p_5^2-Dq_5^2=1$ (for $D\equiv2,3\pmod4$) transparent:
\[
        D=\frac{(p_5-1)(p_5+1)}{q_5^2}=\frac{((a+1)u_0+a)\,((a+3)u_0+(a+2))}{(a+1)(a+3)}.
\]
We shall show that $\alpha_3$ is not a square.  Since $\tau(\alpha_3)=\alpha_1$, and $\alpha_1$ has $\omegaD$-coefficient~$1$, we may apply the coefficient-$1$ calculation from the proof of Proposition~\ref{prop:negative-unit}, which does not depend on the parity of $s$: it shows that $\alpha_1$ is not a square for $D\notin\F$.  Therefore $\alpha_3$ is not $\eps$ times a square either.  Also, $\alpha_3$ itself is indecomposable by Theorem~\ref{thm:indecomp}.

By Lemma~\ref{lem:BK11}, when $\alpha_3$ is a square, its root is---up to sign and conjugation---a convergent.  A conjugate root has square with negative $\omegaD$-coefficient, so it cannot give $\alpha_3$.  Also $\alpha_{-1}^2=1\ne\alpha_3$.

For $D\equiv2,3\pmod4$, the $\omegaD$-coefficient of $\alpha_m^2$ is $2p_mq_m$.  Explicitly, $p_2=(a+1)u_0+a$ and $q_2=a+1$, so
\[
 2p_2q_2 = 2((a+1)u_0+a)(a+1) \ge 2(2a+1)(a+1) > a+2
\]
for $u_0\ge1$ and $a\ge1$.  Also $p_m$ is strictly increasing and $q_m$ is positive and nondecreasing for $m\ge0$ as the recurrence has positive partial quotients---so the product $p_mq_m$ is strictly increasing.  Hence $2p_mq_m>a+2$ for all $m\ge2$, so we must have $m\in\{0,1\}$.  The case $m=0$ ($\alpha_3=\alpha_0^2$) gives, from the coefficient and constant equations,
\[
        a=2u_0-2,
        \qquad
        D=u_0^2+2u_0-1.
\]
Substituting $a=2u_0-2$ into the unit equation displayed above yields a second value of $D$ differing from $u_0^2+2u_0-1$ by
\[
        \frac{4u_0^2-4u_0-1}{(2u_0-1)(2u_0+1)},
\]
which is never zero for integral $u_0$, a contradiction.  The case $m=1$ gives $a=2u_0$ and then
\[
        D=u_0^2+2u_0=(u_0+1)^2-1\in\F.
\]

For $D\equiv1\pmod4$, write
\[
        \alpha_m=(p_m-q_m)+q_m\omegaD.
\]
The $\omegaD$-coefficient of $\alpha_m^2$ is $q_m(2p_m-q_m)$.  Put $h_m=2p_m-q_m$; then $h_{m+1}=u_{m+1}h_m+h_{m-1}$, with positive terms, so $h_m$ is strictly increasing, while $q_m$ is positive and nondecreasing.  Hence $q_mh_m=q_m(2p_m-q_m)$ is strictly increasing.  Concretely,
\[
 q_2(2p_2-q_2) = (a+1)\bigl(2(a+1)u_0 + a - 1\bigr) \ge (a+1)(3a+1) > a+2
\]
for $u_0\ge1$ and $a\ge1$, and the quantity is strictly increasing in $m\ge0$.  Thus again $m\in\{0,1\}$.

Now
\[
        \alpha_0=(u_0-1)+\omegaD,
        \qquad
        \alpha_1=u_0+\omegaD,
\]
so the corresponding square coefficients are $2u_0-1$ and $2u_0+1$.  Matching the coefficient $a+2$ of $\alpha_3$ gives
\[
        a=2u_0-3
        \quad\text{or}\quad
        a=2u_0-1.
\]
For $D\equiv1\pmod4$ the unit equation $\N(\alpha_5)=1$ takes the form
\[
        D=\frac{(2p_5-q_5)^2-4}{q_5^2}=\frac{(2au_0+a+2u_0-1)(2au_0+a+6u_0+1)}{(a+1)(a+3)};
\]
if $a=2u_0-1$, we have
\[
        D=(2u_0-1)(2u_0+3)=(2u_0+1)^2-4\in\F.
\]
If $a=2u_0-3$, the constant equation for $\alpha_3=\alpha_0^2$ gives
\[
        D=4u_0^2+4u_0-7,
\]
whereas the unit equation gives a value differing from this by
\[
        \frac{2(2u_0^2-4u_0+1)}{u_0(u_0-1)};
\]
since $a\ge1$ implies $u_0\ge2$, this difference is never zero.  Hence $\alpha_3$ is not a square.  Thus $\alpha_3$ lies in neither square class, and Lemma~\ref{lem:indecomp-square} gives the obstruction.
\end{proof}

\subsubsection{Complete quotients and even roots}

For periods $s\equiv2\pmod4$ with $s\ge10$ we replace explicit computation by a size argument in the absolute norm $|\N(\alpha_i)|$.  Two classical estimates are in play.  The upper bound $|\N(\alpha_i)|<\delta$, with $\delta=\sqrt{\Delta_K}$, goes back to Dress--Scharlau~\cite{DressScharlau1982}; and the lower bound $|\N(\alpha_i)|>\delta/3$, valid whenever the next partial quotient equals~$1$---in the residual shape, at every even index---follows from \cite[Lemma~8]{TinkovaVoutier2020}.  Our next lemma recovers both from a single closed expression for $|\N(\alpha_i)|$ in terms of the complete quotient $\theta_{i+1}$, a formula related to that of \cite[Proposition~5]{Kala2016a}.  The gap between $\delta/3$ and $\delta$ is what makes squares among the convergents impossible to sustain.

We let
\[
        \theta_{i+1}=[u_{i+1};u_{i+2},\ldots]
\]
be the complete quotient following the $i$-th convergent.

\begin{lemma}\label{lem:complete-quotient}
For every $i\ge0$,
\[
        |\N(\alpha_i)|=\frac{\delta}{|\theta_{i+1}-\theta_{i+1}'|}.
\]
Moreover, we have
\[
        \theta_{i+1}>1,
        \qquad
        -1<\theta_{i+1}'<0,
\]
and so consequently $|\N(\alpha_i)|<\delta$ for all $i\ge0$.  If $u_{i+1}=1$, then
\[
        |\N(\alpha_i)|>\frac{\delta}{3}.
\]
\end{lemma}

\begin{proof}
The standard complete-quotient identity gives
\[
        \theta_{i+1}=-\frac{\alpha_{i-1}'}{\alpha_i'},
\]
equivalently $\theta_{i+1}\alpha_i'=-\alpha_{i-1}'$; see Perron \cite[Chapter II, \S 13]{Perron1913}, and the discussion of complete quotients of Blomer--Kala \cite[\S 2]{BlomerKala2018}.  Conjugating gives
\[
        \theta_{i+1}'=-\frac{\alpha_{i-1}}{\alpha_i}.
\]
Therefore, we have
\[
\theta_{i+1}-\theta_{i+1}'
 =\frac{\alpha_{i-1}\alpha_i'-\alpha_{i-1}'\alpha_i}{\N(\alpha_i)}.
\]
The numerator equals $(-1)^{i-1}(\omegaD-\omegaD')$ by the determinant identity for $p_i,q_i$; taking absolute values proves the claimed formula for $|\N(\alpha_i)|$.

Since $\theta_{i+1}=[u_{i+1};u_{i+2},\ldots]$ has positive partial quotients, $\theta_{i+1}>1$.  For the base case $i=0$ we have $\alpha_{-1}=1$ and $\alpha_0=u_0-\omegaD'>1$ in both congruence classes of $D$, so $\alpha_0>\alpha_{-1}$; inductively, the recurrence $\alpha_{i+1}=u_{i+1}\alpha_i+\alpha_{i-1}$ shows that $\alpha_i>\alpha_{i-1}>0$ for all $i\ge0$, so $-1<\theta_{i+1}'<0$.  Thus $\theta_{i+1}-\theta_{i+1}'>1$, giving $|\N(\alpha_i)|<\delta$.  If $u_{i+1}=1$, then $1<\theta_{i+1}<2$, and hence $\theta_{i+1}-\theta_{i+1}'<3$, giving the lower bound.
\end{proof}

The two bounds combine into an exclusion principle for even-indexed roots.  If an odd-indexed convergent were the square of an even-indexed one, then squaring would multiply norms, sending a quantity larger than $\delta/3$ to one larger than $\delta^2/9$---but the upper bound caps it at $\delta$, forcing $\delta<9$. A short finite check then shows that no residual field is that small.

\begin{lemma}[Even-root exclusion]\label{lem:even-root}
Assume the residual shape $u_1=u_3=\cdots=u_{s-1}=1$, with $s\equiv2\pmod4$ and $s\ge10$, and put $t=s/2$.  Then no convergent $\alpha_{2j-1}$ with $1\le j\le t-1$ is the square of an even-indexed convergent.
\end{lemma}

\begin{proof}
Suppose, for some $1\le j\le t-1$, that
\[
        \alpha_{2j-1}=\alpha_{2r}^2.
\]
Since every even-indexed convergent is followed by an odd-indexed partial quotient equal to $1$, Lemma~\ref{lem:complete-quotient} gives
\[
        |\N(\alpha_{2r})|>\frac{\delta}{3}.
\]
Also $|\N(\alpha_{2j-1})|<\delta$.  Hence
\[
        \delta>|\N(\alpha_{2j-1})|=|\N(\alpha_{2r})|^2>\frac{\delta^2}{9},
\]
so $\delta<9$.

It remains to check that no such residual field has $\delta<9$.  The following tables, obtained by the standard continued-fraction algorithm applied to $\omegaD$, exhaust the squarefree $D$ with $\delta<9$.  If $D\equiv2,3\pmod4$, then $\delta=2\sqrt D<9$, so $D<81/4$, and hence $D\le20$:
\[
\begin{array}{c|c}
D&\text{period of }\omegaD\\ \hline
2&(2)\\
3&(1,2)\\
6&(2,4)\\
7&(1,1,1,4)\\
10&(6)\\
11&(3,6)\\
14&(1,2,1,6)\\
15&(1,6)\\
19&(2,1,3,1,2,8).
\end{array}
\]
If $D\equiv1\pmod4$, then $\delta=\sqrt D<9$, so $D<81$:
\[
\begin{array}{c|c||c|c}
D&\text{period of }\omegaD&D&\text{period of }\omegaD\\ \hline
5&(1)&29&(5)\\
13&(3)&33&(2,1,2,5)\\
17&(1,1,3)&37&(1,1,5)\\
21&(1,3)&41&(1,2,2,1,5)\\
53&(7)&69&(1,1,1,7)\\
57&(3,1,1,1,3,7)&73&(1,3,2,1,1,2,3,1,7)\\
61&(2,2,7)&77&(1,7);\\
65&(1,1,7)&&
\end{array}
\]
none has period length $s\equiv2\pmod4$ with $s\ge10$ and all odd-indexed partial quotients equal to $1$.
\end{proof}

With even roots excluded, the largest periods succumb to a descent on norms: the smallest-norm odd convergent in the interior cannot be a square, because its square root would either be an excluded even convergent or an odd convergent of still smaller norm.

\begin{proposition}\label{prop:residual-large}
Under the standing assumptions of Section~\ref{sec:residual}, if $s\equiv2\pmod4$ and $s\ge10$, then $f$ is not universal.
\end{proposition}

\begin{proof}
Put $t=s/2$, so $t$ is odd, and write $\beta_j=\alpha_{2j-1}$ for $0\le j\le t$.  Then $\tau(\beta_j)=\beta_{t-j}$.  Among $1\le j\le t-1$, choose $j$ with $|\N(\beta_j)|$ minimal.  The norm $|\N(\beta_j)|$ is greater than $1$, as otherwise $\beta_j$ would be a totally positive unit strictly between $1$ and the fundamental totally positive unit $\eps=\beta_t$.

We claim that neither $\beta_j$ nor $\beta_{t-j}$ is a square.  If, say, $\beta_j=v^2$, then Lemma~\ref{lem:BK11} gives $v=\pm\alpha_m$ or $\pm\alpha_m'$ for some $m\ge-1$.  The case $m=-1$ gives $\beta_j=1$, impossible.  For $m\ge0$, a conjugate root has square with negative $\omegaD$-coefficient, so $v=\pm\alpha_m$.  Since $1<\beta_j<\beta_t=\eps=\alpha_{s-1}$ and the sequence $(\alpha_m)_{m\ge-1}$ is strictly increasing, the equality $\beta_j=\alpha_m^2$ forces $\alpha_m<\alpha_{s-1}$, hence $m\le s-2$.  By Lemma~\ref{lem:even-root}, $m$ is not even; combined with $m\le s-2$ and $s$ even, this gives $m\le s-3$.  Thus $m=2r-1$ for some $1\le r\le t-1$, and
\[
        |\N(\beta_r)|^2=|\N(\beta_j)|.
\]
Since $|\N(\beta_j)|>1$, this gives
\[
        1<|\N(\beta_r)|<|\N(\beta_j)|,
\]
contradicting minimality.  The same argument applies to $\beta_{t-j}$, whose norm equals that of $\beta_j$ because $\tau$ preserves absolute norm.

Therefore $\beta_j$ is not a square, and since $\tau(\beta_j)=\beta_{t-j}$ is not a square, $\beta_j$ is not $\eps$ times a square.  Lemma~\ref{lem:indecomp-square} then gives the universality obstruction.
\end{proof}

\section{Proof of the classification}\label{sec:proof}

All the pieces are now in place; it remains only to check that the cases we have considered cover every $D\notin\F$, following the parity-and-period dichotomy laid out in the introduction.

\begin{proof}[Proof of Theorem~\ref{thm:main}]
The sufficient direction is Proposition~\ref{prop:sufficiency}.

For the converse, assume $D\notin\F$.  If $s$ is odd, then Proposition~\ref{prop:negative-unit} shows that $f$ is not universal.  Hence let $s$ be even.  If $s=2$, then having $u_1=1$ would put $D$ in $\F$ by Proposition~\ref{prop:MDone}; thus $u_1\ge2$, and Proposition~\ref{prop:s2} applies.

Now suppose $s\ge4$.  If some indecomposable has non-square absolute norm, then Lemma~\ref{lem:norm-easy} applies.  Otherwise, Corollary~\ref{cor:odd-u-one} gives
\[
        u_1=u_3=\cdots=u_{s-1}=1.
\]
If $s\equiv0\pmod4$, then Proposition~\ref{prop:s0mod4} applies.  If $s\equiv2\pmod4$, then either $s=6$, handled by Proposition~\ref{prop:s6}, or $s\ge10$, handled by Proposition~\ref{prop:residual-large}.
\end{proof}

\section{Concluding remarks}

The next natural problem is to understand fields with $M_D=2$.  Blomer--Kala's construction gives universal diagonal forms in $16$ variables in that case, but a direct characterization of when smaller fixed shapes are universal appears to require finer control of square classes among indecomposables \cite{TinkovaVoutier2020}.  Our result is, in a sense, dual to the theory of universality criterion sets developed by Kala--Kr\'{a}sensk\'{y}--Romeo \cite{KKR2026} (see also \cite{Conway15,Bhargava15,BH290,KKO2005,EKK2013,KaneLiu2020,KalaPrakash2024,KimPark2025}): a criterion set fixes the ring $\OK$ and produces finitely many totally positive integers whose representation certifies universality of an \emph{arbitrary} form, whereas here we fix one form and vary the field.  We do not produce a criterion set; rather, the two viewpoints meet at the invariant $M_D$, since the fields in our two families are exactly those with a single indecomposable class up to units, which is also the regime in which universality criterion sets are smallest.  It would also be interesting to extend the argument to non-maximal real quadratic orders, where the Dress--Scharlau description still exists but the bookkeeping of complete quotients and primitive norm forms is more delicate.  Finally, the proof illustrates the close relation between universal quadratic forms and additive indecomposables, a theme emphasized in recent work on ranks, density-zero theorems, and Pythagoras-number questions \cite{KYZ,Kala2023survey,KalaTinkova2023,KKP2022}.

\providecommand{\bysame}{\leavevmode\hbox to3em{\hrulefill}\thinspace}
\providecommand{\MR}{\relax\ifhmode\unskip\space\fi MR }
\providecommand{\MRhref}[2]{%
  \href{http://www.ams.org/mathscinet-getitem?mr=#1}{#2}
}
\providecommand{\href}[2]{#2}

\end{document}